\documentclass[reqno,12pt]{article}
\usepackage{amsmath,amsthm,amssymb,bm} 

\newtheorem{theorem}{Theorem}
\newtheorem{lemma}[theorem]{Lemma}

\def\E{{\mathbb E}}
\def\P{{\mathbb P}}

\begin{document}
\title{On the concentration of the chromatic number of random graphs}
\author{
Alex Scott
\thanks{Mathematical Institute, University of Oxford, Andrew Wiles Building, Radcliffe Observatory Quarter, Woodstock Road, Oxford OX2 6GG, UK;
email: scott@maths.ox.ac.uk 
}
\date{}
 }

\maketitle

\begin{abstract}
Let $0<p<1$ be fixed.  Shamir and Spencer proved in the 1980s that 
the chromatic number of a random graph $G\in\mathcal G(n,p)$ is concentrated in an interval of length $\omega(n)\sqrt n$.
In this explanatory note, we give a proof of a result due to Noga Alon, showing that
$\chi(G)$ is concentrated in an interval of length $\omega(n)\sqrt n/\log n$.
\end{abstract}

\section{Introduction} 

How concentrated is the chromatic number $\chi(G)$ of a random graph $G\in\mathcal G(n,p)$?
For constant probability $p\in(0,1)$, Shamir and Spencer \cite{SS87} proved in the 1980s that $\chi(G)$ is concentrated in an interval of length $\omega(n)\sqrt n$.
For sparse random graphs, much stronger concentration results are known: Shamir and Spencer \cite{SS87} showed that for $p<n^{-5/6-\epsilon}$, the chromatic number is concentrated on 5 consecutive integers.  \L uczak \cite{L91} sharpened this to a 2-point concentration result, while  Alon and Krivelevich \cite{AK97} extended 2-point concentration to the larger range $p<n^{-1/2-\epsilon}$.

The aim of this note is to show that \L uczak's approach also works for random graphs with constant density, giving a slight improvement on the concentration result of Shamir and Spencer, from $\omega(n)\sqrt n$ to $\omega(n)\sqrt n/\log n$.    After posting the original (2008) version of this note on the arxiv, it emerged that this was proved independently (and a good bit earlier) by Noga Alon, who included it as an exercise in \cite{AS}.  However, apparently no proof has been published, and so it seems worthwhile to leave this as an explanatory note.

\goodbreak

We will prove the following theorem.

\begin{theorem}\label{main}
Let $0<p<1$ be fixed, and suppose $\omega(n)\to\infty$ as $n\to\infty$.  Then there is a function $h=h(n)$ such that,
for $G\in\mathcal G(n,p)$, with probability $1-o(1)$,
\begin{equation}\label{bound}|\chi(G)-h(n)|<\omega(n)\sqrt{n}/{\log n}.
\end{equation}
\end{theorem}

Here and throughout, $\omega(n)$ refers to any function that tends to $\infty$ as $n\to\infty$.

Let us note that bounding the concentration of the chromatic number does not determine its likely value.  The asymptotic behaviour of $\E \chi(G)$, where $G\in\mathcal G(n,p)$ and $p$ is fixed, was resolved by Bollob\'as \cite{B88}, while \L uczak \cite{L91a} extended this result to a much wider range of values $p=p(n)$ (see also McDiarmid \cite{M90}).  In the sparse case ($p=c/n$), Achlioptas and Naor \cite{AN05} have given an explicit pair of consecutive integers $k, k+1$ such that $\chi(G)\in\{k,k+1\}$ almost surely, and
Coja-Oghlan, Panagiotou, Steger \cite{CPS08} recently proved an explicit three-point concentration result for $p<n^{-3/4-\epsilon}$.

Finally, let us note that, for $p$ fixed, nothing is known from below concerning concentration: it has not even been shown that the chromatic number cannot be concentrated in an interval of constant length!  See Bollob\'as (\cite{B04} and \cite{B01}) for further discussion.

\section{Proof}

The proof will proceed as follows.  As in \L uczak \cite{L91a}, we first define $h(n)$ so that $\P(\chi(G)\le h(n))$ tends to 0 slowly.  Then $\chi(G)>h(n)$ with probability $1-o(1)$, so we need only bound $\chi(G)$ from above.  A martingale argument shows that we can colour all but (a little more than) $\sqrt n$ vertices with $h(n)$ colours, so we try to colour the remaining vertices without using too many new colours.  \L uczak's argument used the local sparsity of $G$; here, in the dense case, we can use a fairly crude greedy algorithm, which shows that any reasonably large set $S$ of vertices can be coloured with $O(|S|/\log n)$ colours.  This will be enough to prove \eqref{bound}.

We first need a simple lemma on independent sets in random graphs.  It is slightly cleaner to phrase it in terms of complete subgraphs.

\begin{lemma}\label{ind}
Let $0<p<1$ be fixed, and suppose $\omega(n)\to\infty$ as $n\to\infty$. 
There is a constant $c=c(p)$ such that,
for $G\in\mathcal G(n,p)$, with probability $1-o(1)$, every subset $W\subset V(G)$ with $|W|>n^{1/3}$ contains a complete subgraph with at least $c\log n$ vertices.
\end{lemma}

\begin{proof}[Proof of Lemma \ref{ind}]
Let $G\in\mathcal G(n,p)$, and suppose that $U\subset V(G)$ has $u\ge n^{1/4}$ vertices.  We use Chernoff's inequality in the form that
$X\sim B(r,p)$ implies $\P(X<rp-t)\le\exp(-t^2/2rp)$.  Thus (with $t=pu^2/5$),
$$\P(e(U)<pu^2/4)\le \exp\left(-(pu^2/5)^2/2p\binom u2\right)<\exp(-pu^2/30).$$
So the probability that there is a subset of size $u$ with fewer than $pu^2/4$ edges is at most
$$\binom nue^{-pu^2/30}\le
\left(\frac{en}{u}\right)^u e^{-pu^2/30}=\left(\frac{en}{u}e^{-pu/30}\right)^u=o(1).$$
It follows that, with probability $1-o(1)$, every subset of $u$ vertices has edge density at least $p/2$; thus every subset of $u'\ge u$ vertices has edge density at least $p/2$ and hence induces a subgraph with maximal degree at least $p(u'-1)/2$.

Suppose that this property holds.
Given $W\subset V(G)$, we choose a complete subgraph greedily: set $W_0=W$ and, for $i\ge 1$, pick $w_i\in W_{i-1}$ with $|\Gamma(w_i)\cap W_{i-1}|$ maximal and set $W_i=\Gamma(w_i)\cap W_{i-1}$, halting with the complete subgraph $\{w_1,\ldots,w_i\}$ as soon as $W_i$ is empty.  The observations above imply that (for $n$ sufficiently large) $|W_i|\ge\frac p2(|W_{i-1}|-1)>p|W_{i-1}|/3$ whenever $|W_{i-1}|\ge u$, and so $|W_i|$ is nonempty provided $(p/2)^i<n^{-1/12}$.  The result follows immediately with $c(p)=-1/12\log(p/2)$.
\end{proof}

\begin{proof}[Proof of Theorem \ref{main}]
We may assume that $\omega(n)=o(\log\log n)$.  Let $G\in \mathcal G(n,p)$ and let
\begin{equation}\label{h}
h(n)=\min\big\{r:\P(\chi(G)\le r)>\frac{1}{\omega(n)}\big\}.
\end{equation}
We shall show that \eqref{bound} holds with this $h$.

Clearly $\P(\chi(G)<h(n))\to0$ as $n\to\infty$; thus it suffices to show that, with probability $1-o(1)$, we have $\chi(G)<h(n)+\omega(n)\sqrt{n}/\log n$.
Let
$$s(G)=\max\{|W|: W\subset V(G), \chi(G[W])\le h(n)\}$$
be the maximum number of vertices we can colour with $h$ colours.  Consider the vertex exposure martingale: modifying the edges from a single vertex can change $s(G)$ by at most 1, so the Azuma-Hoeffding inequality implies that, for $t>0$,
$$\P(|s(G)-\E s(G)|>t)\le 2e^{-t^2/n}.$$
In particular, for $n$ sufficiently large,
\begin{equation}\label{hh}
\P(|s(G)-\E s(G)|>\sqrt{\omega(n)n})<\frac{1}{\omega(n)}.
\end{equation}
It follows from \eqref{h} that
$\P(s(G)=n)>1/\omega(n)$.  Thus \eqref{hh} implies that $\E(s(G))\ge n-\sqrt{\omega(n)n}$ and therefore
$\P(s(G)<n-2\sqrt{\omega(n)n})<1/\omega(n)$.  Let $W\subset V(G)$ be a subset of maximal size such that $\chi(G[W])\le h(n)$, and let $U=V(G)\setminus W$.  We claim that, with probability $1-o(1)$, $\chi(G[U])\le\omega(n)\sqrt n/\log n$.  Since $\chi(G)\le\chi(G[W])+\chi(G[U])$,  \eqref{bound}  follows immediately.

The claim follows simply from Lemma \ref{ind}.  Indeed, $|U|\le 2\sqrt{\omega(n)n}$ with probability $1-o(1)$.  Let us greedily remove independent sets of maximal size from $U$, giving a new colour to each, until $n^{1/3}$ vertices remain (and give new colours to each of these).  The lemma (applied to $\overline G$) implies that, with probability $1-o(1)$, we use at most $O(2\sqrt{\omega(n)n}/c(p)\log n +n^{1/3})$ colours, which is bounded by $\omega(n)\sqrt n/\log n$ for sufficiently large $n$. 
\end{proof}

\bigskip

\noindent{\bf Acknowledgements:}  The author would like to thank Noga Alon, Boris Bukh and Benny Sudakov for helpful comments.

\end{document}